\documentclass[10pt]{amsart}

\usepackage{amsmath,amssymb,amscd,amsthm,latexsym,accents,stmaryrd}
\usepackage{amsfonts}
\usepackage[latin9]{inputenc}
\usepackage[all]{xy}
\usepackage{enumerate}

\newtheorem{thm}{Theorem}[section]

\newtheorem{cor}[thm]{Corollary}
\newtheorem{defin}[thm]{Definition}

\newcommand{\cof}{\rightarrowtail }

\newcommand{\we}{\stackrel{\sim}{\rightarrow}}

\newdir{ >}{{}*!/-7pt/\dir{>}}

\def\Q{{\mathbb{Q}}}

\def\Z{{\mathbb{Z}}}
\def\cat{{\rm{cat}\hskip1pt}}
\def\Mcat{{\rm{Mcat}\hskip1pt}}
\def\secat{{\rm{secat}\hskip1pt}}
\def\Msecat{{\rm{Msecat}\hskip1pt}}
\def\Hsecat{{\rm{Hsecat}\hskip1pt}}
\def\nil{{\rm{nil}\hskip1pt}}
\def\TC{{\rm{TC}\hskip1pt}}
\def\MTC{{\rm{MTC}\hskip1pt}}
\def\HTC{{\rm{HTC}\hskip1pt}}

\def\mtc{{\mathbf{mtc}\hskip1pt}}
\def\htc{{\mathbf{htc}\hskip1pt}}
\def\Hom{{\rm{Hom}}}

\begin{document}

\title{Rational approximations of sectional category and Poincar\'e duality}

\author[J.G. Carrasquel-Vera]{Jos\'e Gabriel Carrasquel-Vera}

\address{\begin{sloppypar}J.G. Carrasquel-Vera, Institut de Recherche en Mathématique et Physique, Université Catholique de Louvain, 2 Chemin du Cyclotron, 1348 Louvain-la-Neuve, Belgium\end{sloppypar}}

\email{jose.carrasquel@uclouvain.be}

\author[T. Kahl]{Thomas Kahl}

\address{T. Kahl, Centro de Matem\'atica, Universidade do Minho, Campus de Gualtar, 4710-057 Braga, Portugal}

\email{kahl@math.uminho.pt}

\author[L. Vandembroucq]{Lucile Vandembroucq}

\address{L. Vandembroucq, Centro de Matem\'atica,Universidade do Minho, Campus de Gualtar, 4710-057 Braga, Portugal}

\email{lucile@math.uminho.pt}

\thanks{The research of the second and third authors was supported by FCT -  \emph{Fundação para a Ciência e a Tecnologia} through projects PTDC/MAT/0938317/2008 and  PEstOE/MAT/UI0013/2014.}




\begin{abstract}
Félix, Halperin, and Lemaire have shown that the rational module category $\Mcat$ and the rational Toomer invariant $e_0$ coincide for simply connected Poincaré duality complexes. We establish an analogue of this result for the sectional category of a fibration.
\end{abstract}

\maketitle

\section{Introduction}

The sectional category (or genus) of a fibration $p\colon  E\to X$, $\secat(p)$, is the least integer $n$ such that $X$ can be covered by $n+1$ open sets on each of which $p$ admits a continuous local section. This invariant, which has been introduced by Schwarz in \cite{schwarz}, is a generalization of the Lusternik-Schnirelmann category of a space: if $X$ is a path-connected pointed space, then the L.-S. category of $X$, $\cat(X)$, is precisely the sectional category of the path fibration $ev_1\colon  PX\to X$, $\lambda\to \lambda(1)$ where $PX$ is the space of paths beginning in the base point. Another important special case of sectional category is the topological complexity of a space, introduced by Farber \cite{Farber1} in order to give a topological measure of the complexity of motion planning problems in robotics. The topological complexity of a space $X$, $\TC(X)$, is the sectional category of the fibration $ev_{0,1}\colon X^{[0,1]}\to X\times X$, $\lambda \to (\lambda(0),\lambda(1))$.

L.-S. category has been extensively studied in the framework of rational homotopy theory (see \cite[Part V]{RHT}, \cite[Chapter 5]{CLOT}). It has been established, in particular, that the rational L.-S. category of a simply connected Poincar\'e duality complex
$X$, $\cat_0(X)$, coincides with the rational Toomer invariant of $X$, $e_0(X)$. This result actually follows from two independent results: first, Hess' theorem \cite{Hess} that, for any simply connected CW complex of finite type, the rational category coincides with the invariant $\Mcat$, and second, the theorem by F\'elix, Halperin, and Lemaire \cite{FHL} that, for any simply connected Poincar\'e duality complex $X$, one has $\Mcat(X)=e_0(X)$. The goal of this paper is to extend the result of F\'elix, Halperin, and Lemaire to sectional category. More precisely, considering the generalizations, respectively denoted by $\Msecat$ and $\Hsecat$, of the invariants $\Mcat$  and $e_0$ to fibrations (see Section \ref{Msecat}),  we establish in Section \ref{dem} the following theorem:

\begin{thm}\label{Poincare} Let $p\colon E\to X$ a fibration such that $H^*(X;\Q)$ is a Poincar\'e duality algebra. Then $\Msecat(p)=\Hsecat(p).$
\end{thm} 
In Section \ref{MTC}, we apply this result to Farber's topological complexity. 

All spaces we consider in this article are compactly generated Hausdorff spaces, and fibrations are understood to be surjective. Throughout this paper we work over the field $\Q$ of rational numbers. All graded vector spaces we consider are $\Z$-graded with upper degree, and all differential vector spaces are cochain complexes, i.e., the differential raises the upper degree by one. Our general reference for notions and results from rational homotopy theory is \cite{RHT}.

\emph{Acknowledgement.} We are indebted to Yves Félix from whom we learned a proof of the equality $\Mcat (X) = e_0(X)$ for Poincaré duality complexes that we were able to generalize in order to prove Theorem \ref{Poincare}.

\section{Sectional category, $\Msecat$, and $\Hsecat$}

\subsection{Fundamental results on sectional category}

We recall here two fundamental results on sectional category due to Schwarz \cite{schwarz}.

Let $p\colon E\to X$ be a fibration, and let $p^*\colon H^*(X)\to H^ *(E)$ be the morphism induced by $p$ in  cohomology. Then, considering the nilpotency of the ideal $\ker p^*$, that is, the least integer $n$ such that any $(n+1)$-fold cup product in $\ker p^*$ is trivial, we have

\begin{thm}\label{schwarzthm1} $\nil\ker p^* \leq \secat(p) \leq \cat(X)$.
\end{thm}

Another fundamental result is the very useful characterization of sectional category in terms of joins. 

\begin{defin}\rm 
The \textit{(fiber) join} of two maps $p \colon  E \to X$ and $p'\colon E'\to X$, denoted by $E\ast _XE'$, is the double mapping cylinder of the projections $E\times _XE' \to E$ and $E\times _XE' \to E'$, i.e., the quotient space $((E\times _XE')\times I \amalg E \amalg E')/\sim$ where $(e,e',0) \sim e$, $(e,e',1) \sim e'$. The \textit{join map} of $p$ and $p'$ is the map $j_{p,p'} \colon  E\ast_XE' \to X$ defined by $j_{p,p'}([e,e',t]) = p(e) = p'(e')$, $j_{p,p'}([e]) = p(e)$, and $j_{p,p'}([e']) = p'(e')$. The \textit{$n$-fold join} and the \textit{$n$th join map} of $p$ are iteratively defined by $\ast^0_XE = E $, $\ast^n_XE = (\ast^{n-1}_XE) \ast_XE$, $j^0p = p$, and $j^np = j_{j^{n-1}p,p}$.
\end{defin}

Notice that here the notation $*^n$ means the join of $n+1$ copies of the
considered object.

\begin{thm}\label{joinsecat}
Let $p\colon E\to X$ be a fibration. If $X$ has the homotopy type of a CW complex, then $\secat p \leq n$ if and only if $j^np$ has a section.
\end{thm}

\subsection{DGA modules}
\begin{sloppypar}
Recall that if $(A,d)$ is a differential algebra, a (left) $(A,d)$-module (or differential module over $(A,d)$) is a graded vector space $(M,d)$ with an action of $A$ that is compatible with the differential (${d(a\cdot m)}={da\cdot m}+ {(-1)^{|a|}a\cdot dm}$). If ${\varphi\colon (A,d)\to (B,d)}$ is a morphism of differential algebras, then $(B,d)$ is naturally endowed with an $(A,d)$-module structure, given by ${a\cdot b}={\varphi(a)\cdot b}$, and $\varphi$ is a morphism of $(A,d)$-modules. We say that a morphism $f\colon (M,d) \to (N,d)$ of $(A,d)$-modules admits a \emph{homotopy retraction} if there exists a commutative diagram of $(A,d)$-modules 
\end{sloppypar}
$$\xymatrix{
(M,d)\ar[r]^{id}_{} \ar[d]_{f}_{} \ar[dr]^{}_{}
& (M,d)
\\ 
(N,d)
& (P,d) \ar[l]_-{\sim} \ar[u]^{}_{} 
}$$
where the arrow $\we$ is a quasi-isomorphism.
It is well known that the category of $(A,d)$-modules is a proper closed model category (see for example \cite{mtc}). A morphism of $(A,d)$-modules admits a homotopy retraction if and only if it admits a retraction in the homotopy category of that model category.

\subsection{The invariants $\Msecat$ and $\Hsecat$} \label{Msecat}
Let $p: E \to X$ be a fibration. As usual, we denote by $A_{PL}$ Sullivan's (contravariant) functor of polynomial forms from the category of spaces to the category of commutative cochain algebras.  In \cite{mtc}, the \emph{module sectional category} of $p$, $\Msecat(p)$, has been defined as the least integer $n$ such that the morphism of $A_{PL}(X)$-modules $A_{PL}(j^np)\colon A_{PL}(X) \to A_{PL}(\ast^n_XE)$ admits a homotopy retraction.
If no such $n$ exists, one sets $\Msecat (p) = \infty$. If $p$ is a fibration over a space of the homotopy type of a CW complex, then, by Theorem \ref{joinsecat},  $\Msecat(p)\leq \secat(p)$.

Consider a simply connected pointed space $X$ of the homotopy type of a CW complex  of finite type and the path fibration $ev_1: PX \to X$, $\lambda \mapsto \lambda(1)$. It has been shown in \cite{mtc} that $\Msecat (ev_1)$ equals the classical invariant $\Mcat (X)$. As mentioned in the introduction, K. Hess \cite{Hess} established that $\Mcat(X)=\cat_0(X)$, the rational L.-S. category of $X$. An example showing that, in general, $\Msecat$ does not coincide with rational sectional category is given in \cite{mtc,Stanley}.

We define the \textit{cohomology sectional category} of $p$, $\Hsecat(p)$, by
$$\begin{array}{rcl}
\Hsecat(p)\leq n & :\Leftrightarrow & H(A_{PL}(j^ n(p))) \mbox{ is injective}\\
& \Leftrightarrow & H^*(j^ n(p)) \mbox{ is injective.}
\end{array}
$$ 
We obviously have 
$\Hsecat(p)\leq \Msecat(p)$. If $X$ is a simply connected well-pointed space of the homotopy type of a CW complex and  $p=ev_1\colon PX\to X$, then $\Hsecat(p)$ coincides with the rational Toomer invariant of $X$, $e_0(X)$. In \cite{mtc}, it has been established that $\Msecat(p) \geq \nil\ker p^*$. We note that the argument given in \cite{mtc} actually shows that  $\Hsecat(p) \geq \nil\ker p^*$.
We finally also observe that the existence of a commutative diagram of the form
$$\xymatrix{\ast^n_XPX \ar[rr]\ar[rd]_{j^nev_1} && \ast^n_X E \ar[ld]^{j^np}\\
&X
}$$
permits one to establish the following analogues of the second inequality of Theorem \ref{schwarzthm1} for a fibrations over simply connected CW complexes of finite type:
$$\Hsecat(p)\leq e_0(X), \quad \Msecat(p)\leq \Mcat(X).$$

\section{Fibrations over Poincaré duality spaces} \label{dem}

\subsection{Poincar\'e duality}\label{Poincareduality}

Recall that a finite dimensional commutative graded algebra $H$ with $H^0=\Q$ is a (rational) \emph{Poincar\'e duality algebra}  of \emph{formal dimension} $n$ if $H$ is concentrated in degrees $0\leq p\leq n$ and there exists an element $\Omega \in H^ n$ such that  $H^n={\Q\cdot \Omega}$ and the map of degree $-n$
$$\begin{array}{rcl}
\Phi\colon H &\to& \Hom(H,\Q) \\
a & \mapsto &\Phi(a)\colon b\mapsto \Omega^{\sharp}(a\cdot b)
\end{array},
\quad \Omega^{\sharp}(\Omega)=1,\,\, \Omega^{\sharp}(H^ {\not=n})=0
$$
is an isormophism. The element $\Omega$ and the isomorphism $\Phi$ are respectively referred to as the \emph{fundamental class} and the \emph{Poincar\'e duality isomorphism}.

If $(A,d)$ is a commutative cochain algebra such that $H=H(A,d)$ satisfies Poincar\'e duality, then a map analogous to $\Phi$ can be defined on $A$. Indeed, let $\omega \in A^n$ be a cocycle representing the fundamental class $\Omega \in H^ n$, and let $S\subset A^n$ be a subset such that
$$A^n=\Q\cdot \omega \oplus S \quad \mbox{and} \quad d(A^{n-1})\subset S.$$
We can then define $\omega^{\sharp}\colon A\to \Q$ by
$$\omega^{\sharp}(\omega)=1, \quad \omega^{\sharp}(S)=0, \quad \mbox{and} \quad \omega^{\sharp}(A^i)=0 \,\, \mbox{ for } i\neq n$$
as well as the following map of degree $-n$:
$$\begin{array}{rcl}
\phi\colon A &\to& \Hom(A,\Q) \\
a & \mapsto &\phi(a)\colon b\mapsto \omega^{\sharp}(a\cdot b)
\end{array}.
$$
Recall that $\Hom(A,\Q)$ is an $(A,d)$-module with respect to the action of $A$ given by ${a\cdot f(x)}={(-1)^{|a||f|}f(ax)}$ and the differential $\delta$ given by ${\delta f}={-(-1)^ {|f|}fd}$. It is a  well-known fact that the map $\phi$ is a quasi-isomorphism of $(A,d)$-modules.

\subsection{Proof of Theorem \ref{Poincare}} 

Let $p\colon E\to X$ a fibration such that  $H^*(X)$ is a Poincar\'e duality algebra. Recall that we want to prove that $\Msecat(p)=\Hsecat(p)$. We already know that $\Msecat(p)\geq\Hsecat(p)$. The other direction follows by applying the following theorem to the morphism of commutative cochain algebras $A_{PL}(j^np)\colon A_{PL}(X) \to A_{PL}(\ast^n_XE)$.

\begin{thm}
Let $f\colon  (A,d)\rightarrow (B,d)$ be a  morphism of commutative cochain algebras such that $H(A,d)$ is a Poincar\'e duality algebra. If $H(f)$ is injective, then $f$ admits a homotopy retraction as a morphism of $(A,d)$-modules.
\end{thm}

\begin{proof}
We first use the so-called ``surjective trick" \cite[p. 148]{RHT} to factor $f$ in the category of commutative cochain algebras in a quasi-isomorphism  $i\colon (A,d) \to (R,d)$ with a retraction $r$ and a surjective morphism $q\colon (R,d) \to (B,d)$. We then have $B = 
R/K$ where $K$ is the kernel of $q$. Since $H(f)$ is injective, so is
$H(q)$, and since $H(A,d)$ is a Poincar\'e duality algebra, so is $H(R,d)$. As in Section \ref{Poincareduality}, denote by $\omega\in R^n$ a cocycle representing the fundamental class of $H(R,d)$. Since $H(q)$ is injective, $\omega\notin d(R^{n-1})+K^n$. Thus we can write 
$$R^n=\Q\cdot \omega\oplus S$$ 
where $S$ satisfies $d(R^{n-1})+K^n\subset S$. Observe that, under these conditions, the map $\omega^{\sharp}$ defined in Section \ref{Poincareduality} satisfies $\omega^{\sharp}(K)=0$. Since $K$ is an ideal of $R$, the quasi-isomorphism of $(R,d)$-modules  $\phi\colon (R,d)\to (\Hom(R,\Q),\delta)$ (of degree $-n$) defined in Section  \ref{Poincareduality} satisfies $\phi(K)=0$. Therefore, composing with the suspension isomorphism, we get a quasi-isomorphism of $(R,d)$-modules (of degree 0) 
$$\hat{\phi}\colon  (R,d) \stackrel{\phi}{\longrightarrow}(\Hom(R,\Q),\delta)\stackrel{\cong}{\longrightarrow} s^{-n}(\Hom(R,\Q),\delta)$$ 
that factors as $q$ followed by a morphism of $(R,d)$-modules ${l\colon (B,d) = (R/K,d) }\to s^{-n}(\Hom(R,\Q),\delta)$. Recall that the $n$th suspension of an $(R,d)$-module $(M,d)$ is the $(R,d)$-module $s^{-n}(M,d) = (s^{-n}M,d)$ given by $(s^{-n}M)^i=M^{i-n}$, ${a\cdot s^{-n}x}={(-1)^{n|a|}s^{-n}(a\cdot x)}$ and ${d(s^{-n}x)}={(-1)^ns^{-n}(dx)}$. The suspension isomorphism given by $x\mapsto s^{-n}x$ is an isomorphism of $(R,d)$-modules of degree $n$. 

Now factor $q$ in the closed model category of $(R,d)$-modules in a cofibration $j\colon (R,d) \cof (P,d)$ and a weak equivalence $\psi\colon (P,d) \we (B,d)$ and form the following commutative diagram of $(R,d)$-modules: 
\[\xymatrix{
(R,d)\ar[rr]^{id}\ar@{ >->}[d]_{j}&&(R,d)\ar[d]^\sim_{\hat{\phi}}\\
(P,d)\ar[r]_{\psi}^{\sim} & (B,d)\ar[r]_(.35){l}&s^n( Hom(R,\Q),\delta)
}\]
By the well-known lifting lemma, we obtain a retraction $\rho$ of $j$. The commutative diagram of $(A,d)$-modules 
\[\xymatrix{
(A,d) \ar@/^2pc/[rrr]^{id} \ar[r]^{\sim}_{i} \ar[d]_{f} & (R,d)  \ar[r]^{id}  \ar[d]^{j} & (R,d) \ar[r]^{\sim}_{r} & (A,d)\\
(B,d)   & (P,d) \ar[l]^{\sim}_{\psi} \ar[ur]_{\rho} & &
}\]
shows that $f$ admits a homotopy retraction of $(A,d)$-modules.
\end{proof}

\section{Application to topological complexity} \label{MTC}

\subsection{Topological complexity and the lower bounds $\MTC$ and $\HTC$} 

Let $X$ be a path-connected space of the homotopy type of a CW complex of finite type. The \emph{topological complexity} of  $X$, $\TC(X)$, is the sectional category of the fibration ${ev_{0,1}\colon X^{[0,1]}}\to {X\times X}$, $\lambda \to (\lambda(0),\lambda(1))$ \cite{Farber1}. Note that this fibration is equivalent to the diagonal map ${\Delta\colon X}\to {X\times X}$. Therefore, the cohomological lower bound ${\nil\ker ev_{0,1}^*}={\nil\ker \Delta^*}$ coincides with ${\nil\ker\cup}$ where $\cup\colon {H^ *(X)\otimes H^ *(X)} \to H^* (X)$ is the cup product. Theorem \ref{schwarzthm1} gives:
$$\nil\ker \cup\leq \TC(X)\leq \cat(X\times X).$$
One also has  $\TC(X)\geq \cat(X)$ when $X$ is pointed \cite{Farber1}. Using the invariants $\Msecat$ and $\Hsecat$, we obtain two other lower bounds of $\TC$, for which we naturally use the following notations:
$$\HTC(X):=\Hsecat(ev_{0,1}), \quad  \MTC(X):=\Msecat(ev_{0,1}).$$ The invariant $\MTC$ has been introduced in \cite{mtc}. We have 
$$\nil \ker \cup \leq \HTC(X) \leq \MTC(X).$$ As a consequence of Theorem \ref{Poincare}, we obtain

\begin{cor} If $H^*(X)$ is a Poincar\' e duality algebra, then $\MTC(X)=\HTC(X)$.
\end{cor}

\begin{proof}
It suffices to observe that since $H^*(X)$ is a Poincar\' e duality algebra, this also holds for ${H^*(X\times X)} = {H^*(X)\otimes H^*(X)}$.
\end{proof}

If $H^*(X)$ is not a Poincar\' e duality algebra, the numbers $\MTC(X)$ and $\HTC(X)$ can be different. Indeed, Félix, Halperin, and Thomas \cite{FHT83} constructed a simply connected CW complex of finite type $X$ such that $\cat_0(X) = \infty$ and $e_0(X) =2$. It has been shown in \cite{mtc} that $\MTC(X)\geq  \cat_0(X)$, and so it follows that $\MTC(X) = \infty$. On the other hand, $\HTC(X) \leq e_0(X\times X) = 2e_0(X) = 4$. Let us also note that \cite{mtc} contains an example showing that $\MTC$ can be greater than $\nil\ker\cup$. The argument given actually shows that this also holds for $\HTC$.

\subsection{The invariants of Jessup,  Murillo, and Parent}

Let $X$ be a simply connected space of the homotopy type of a CW complex of finite type, and let $(\Lambda (V),d)$ be a Sullivan model of $X$. Inspired by the classical algebraic description of rational L.-S. category due to F\'elix and Halperin \cite{FH}, Jessup,  Murillo, and Parent \cite{JMP} consider the multiplication ${\mu\colon (\Lambda V \otimes \Lambda V,d)}\to {(\Lambda V,d)}$, which is a model of the fibration ${ev_{0,1}\colon X^ I}\to {X\times X}$, and define the invariants $\mathbf{tc}(X)$ and $\mathbf{mtc}(X)$ in terms of the projections 
$$p_n\colon (\Lambda V \otimes \Lambda V,d) \to \left(\frac{\Lambda V \otimes \Lambda V}{(\ker \mu)^{n+1}},\bar d\right)$$
as follows: 
\begin{itemize}
\item[-] $\mathbf{tc}(X)$ is the least integer such that $p_n$ admits a homotopy retraction as a morphism of commutative cochain algebras;
\item[-] $\mathbf{mtc}(X)$ is the least integer such that $p_n$ admits a homotopy retraction as a morphism of $(\Lambda V \otimes \Lambda V,d)$-modules.
\end{itemize}
They establish that
$$\TC_0(X)\leq \mathbf{tc}(X) \quad \mbox{and} \quad \MTC(X)\leq \mathbf{mtc}(X)\leq \mathbf{tc}(X).$$
Here, $\TC_0(X)$ is the rational topological complexity of $X$, i.e., the topological complexity of a rationalization of $X$. 

We can extend the above definitions by denoting by $\htc(X)$ the least integer $n$ such that $H(p_n)$ is injective. We then have $\HTC(X)\leq \htc(X)$. Indeed, as shown in \cite{mtc}, the $n$th join map of the fibration ${ev_{0,1}\colon X^I}\to {X\times X}$ can be modeled by a semifree extension of ${(\Lambda V\otimes \Lambda V,d)}$-modules of the following form:
$$(\Lambda V\otimes \Lambda V,d)\to (\Lambda V\otimes \Lambda V\otimes (\Q\oplus Y),d), \qquad d(Y)\subset (\ker\mu)^{n+1} \oplus \Lambda V\otimes \Lambda V\otimes Y.$$
We can therefore define a morphism $$\xi\colon  (\Lambda V\otimes \Lambda V\otimes (\Q\oplus Y),d)\to \left(\frac{\Lambda V \otimes \Lambda V}{(\ker \mu)^{n+1}},\bar d\right)$$
of $(\Lambda V\otimes \Lambda V,d)$-modules by setting $\xi(Y)=0$. This morphism permits us to see that $\MTC(X)\leq  \mathbf{mtc}(X)$ and that $\HTC(X)\leq  \mathbf{htc}(X)$.

Applying Theorem \ref{Poincare} to the morphism $p_n$, we get:
\begin{cor}
If $H^*(X)$ is a Poincar\' e duality algebra, then $\htc(X)=\mtc(X)$.
\end{cor}

A natural question is whether the numbers $\mathbf{tc}(X)$, $\mathbf{mtc}(X)$, and $\mathbf{htc}(X)$ coincide with $\TC_0(X)$, $\MTC(X)$, and $\HTC(X)$ respectively. Although this is true for some classes of spaces, such as, for instance, 
\begin{itemize}
\item[-] formal spaces (in this case, all invariants coincide with $\nil\ker\cup$)
\item[-] spaces with finite dimensional rational homotopy concentrated in odd degrees (in this case, all invariants coincide with $\cat_0(X)$),
\end{itemize}
we do not know if it is true in general.

\end{document}